\documentclass[reqno, a4paper, 12pt]{amsart}
\usepackage[utf8]{inputenc}
\usepackage[T1]{fontenc}
\usepackage[english]{babel}
\usepackage[babel]{microtype}
\usepackage[headings]{fullpage}
\usepackage{amsmath,amssymb,amsthm,amsrefs}
\usepackage{mathrsfs,mathtools}
\usepackage{bbm}
\usepackage{dsfont}
\usepackage{lmodern}
\usepackage{cases}
\usepackage{enumitem}
\usepackage{centernot}
\usepackage{hyperref}
\usepackage{algpseudocode}
\usepackage{cleveref}

\DeclareFontShape{OMX}{cmex}{m}{n}{%
  <-7.5> cmex7
  <7.5-8.5> cmex8
  <8.5-9.5> cmex9
  <9.5-> cmex10
}{}
\DeclareSymbolFont{largesymbols}{OMX}{cmex}{m}{n} 

\numberwithin{equation}{section}
\newtheorem{theorem}    {Theorem}[section]
\newtheorem{lemma}      [theorem] {Lemma}
\newtheorem{conjecture} [theorem] {Conjecture}

\newtheorem{definition} [theorem] {Definition}
\newtheorem{proposition}[theorem] {Proposition}

\newtheorem{problem}    [theorem] {Problem}

\setlist[enumerate,1]{label=(\roman*)}

\newcommand{\comp}[1]{#1^c}

\DeclareMathOperator{\defined}{\coloneqq}
\DeclareMathOperator{\from}{\colon}

\newcommand{\euler}{e}
\DeclareMathOperator{\prob}{\mathbb{P}}
\DeclareMathOperator{\expec}{\mathbb{E}}

\DeclareMathOperator{\inj}{inj}

\newcommand{\arrows}[1]{\stackrel{#1}{\rightarrow}}
\newcommand{\notarrows}[1]{\centernot{\stackrel{#1}{\rightarrow}}}
\newcommand{\barrows}[1]{\stackrel{#1}{\hookrightarrow}}
\newcommand{\notbarrows}[1]{\centernot{\stackrel{#1}{\hookrightarrow}}}
\newcommand{\Ram}[2]{\mathrm{r}_{#1}(#2)}
\newcommand{\RamN}[2]{\mathrm{R}_{#1}(#2)}
\newcommand{\BipRamN}[2]{\mathrm{B}_{#1}(#2)}
\newcommand{\BlowRam}[4]{\mathrm{B}\big(#2 \arrows{#1} #3 ; #4 \big)}
\newcommand{\SizeRam}[2]{\hat{\mathrm{r}}_{#1}(#2)}
\newcommand{\Gnp}[2]{G(#1,#2)}

\newcommand{\Mult}[3]{M_{#1}(#2;#3)}
\newcommand{\RamRob}[3]{\beta_{#1}(#2;#3)}

\begin{document}

\title{Blowup Ramsey numbers}

\author{Victor Souza}
\address{IMPA, Estrada Dona Castorina 110,
Jardim Botânico, Rio de Janeiro, 22460-320, Brazil}
\email{souza@impa.br}


\begin{abstract}
We study a generalisation of the bipartite Ramsey numbers to blowups of graphs.
For a graph $G$, denote the $t$-blowup of $G$ by $G[t]$.
We say that $G$ is $r$-Ramsey for $H$,
and write $G \stackrel{r}{\rightarrow} H$,
if every $r$-colouring of the edges of $G$
has a monochromatic copy of $H$.
We show that if $G \stackrel{r}{\rightarrow} H$,
then for all $t$,
there exists $n$ such that $G[n] \stackrel{r}{\rightarrow} H[t]$.
In fact, we provide exponential lower and upper bounds
for the minimum $n$ with $G[n] \stackrel{r}{\rightarrow} H[t]$,
and conjecture an upper bound of the form $c^t$,
where $c$ depends on $H$ and $r$,
but not on $G$.
We also show that this conjecture holds for $G(n,p)$
with high probability,
above the threshold for the event $G(n,p) \stackrel{r}{\rightarrow} H$.
\end{abstract}


\clearpage\maketitle
\thispagestyle{empty}


\section{Introduction}
\label{sec:intro}

We say that a graph $G$ is $r$-Ramsey for a graph $H$,
and write $G \arrows{r} H$,
if every $r$-colouring of the edges of $G$
contains a monochromatic copy of $H$.
The classical theorem of Ramsey~\cite{ramsey1930original} from 1930
shows that for every $t$,
there is an $n$ such that {$K_n \arrows{r} K_t$}.
The smallest $n$ with this property is called the diagonal $r$-Ramsey number
and is denoted $\RamN{r}{t}$.
It was proved by Erd\H{o}s and Szekeres~\cite{erdos_szekeres1935combinatorial}
and by Erd\H{o}s~\cite{erdos1947some_remarks}
that $2^{t/2} \leq \RamN{2}{t} \leq 4^t$.
Currently, the best bounds are
\begin{equation*}
  \bigl( 1 + o(1) \bigr)
  \bigl(\sqrt{2}/\euler\bigr) \, t 2^{t/2}
    \leq \RamN{2}{t}
    \leq e^{-c (\log t)^2} 4^t,
\end{equation*}
for some constant $c > 0$.
The lower bound,
due to Spencer~\cite{spencer1975ramsey_lower},
is an application of the Lovász Local
Lemma~\cite{erdos_lovasz1975problems}.
The upper bound was due to Sah~\cite{sah2020ramsey},
improving on previous results of
Rödl~\cite{graham_rodl1987ramsey_numbers},
Thomason~\cite{thomason1998upper_ramsey}
and Conlon~\cite{conlon2009diagonal_ramsey}.

The first results on the bipartite analogue of the Ramsey numbers
were proved by
Beineke and Schwenk~\cite{beineke_schwenk1976bipartite_ramsey} in 1975.
The bipartite $r$-Ramsey number $\BipRamN{r}{t}$
is defined to be the smallest $n$
such that {$K_{n,n} \arrows{r} K_{t,t}$}.
The best current bounds for these numbers,
in the case $r = 2$, are
\begin{equation*}
  \bigl(1+o(1)\bigr)
  \bigl(\sqrt{2}/\euler\bigr)\, t 2^{t/2}
    \leq \BipRamN{2}{t}
    \leq \bigl(1+o(1)\bigr) \log_2(t) 2^{t+1},
\end{equation*}
where the lower bound
is due to Hattingh and Henning~\cite{hattingh_henning1998bipartite_ramsey},
and the upper bound to Conlon~\cite{conlon2008bipartite_ramsey}.

In this paper,
we consider a generalisation of the bipartite Ramsey numbers
to blowups of general graphs.
We denote by $G[t]$ the $t$-blowup of $G$
(see \Cref{sec:upper_bound} for a precise definition)
and call a copy of $H[t]$ in $G[n]$ \emph{canonical}
if it is the $t$-blowup of a copy of $H$ in $G$.
We say that $G[n]$ is \emph{canonically $r$-Ramsey} for $H[t]$,
and write {$G[n] \barrows{r} H[t]$},
if every $r$-colouring of the edges of $G[n]$
has a canonical monochromatic copy of $H[t]$.
Define
\begin{equation*}
  \BlowRam{r}{G}{H}{t}
  = \min \bigl\{ n : G[n] \barrows{r} H[t] \bigr\},
\end{equation*}
as the \emph{blowup Ramsey numbers}.
This generalises the bipartite Ramsey numbers,
since every copy of $K_2[t]$ in $K_2[n]$ is canonical,
so {$\BipRamN{r}{t} = \BlowRam{r}{K_2}{K_2}{t}$}.

A necessary condition for these numbers to be finite
is that {$G \arrows{r} H$}.
Indeed, if {$G \notarrows{r} H$},
consider a colouring of $G$ without monochromatic copies of $H$.
Taking the $n$-blowup of this colouring,
we see that {$G[n] \notbarrows{r} H$} for all $n$.
Assuming that {$G \arrows{r} H$},
one can obtain a bound on {$\BlowRam{r}{G}{H}{t}$} in the following way.
Let $n$ be sufficiently large
and consider an $r$-colouring of the edges $G[n]$.
Repeatedly apply the bipartite Ramsey theorem
between vertex classes in $G[n]$ corresponding to edges of $G$.
Each time, restrict the vertex classes
to contain only vertices used by the monochromatic bipartite graph we obtain.
Doing this for each of the $e(G)$ pairs of vertex classes with edges between,
we obtain a canonical copy of $G[t]$ in $G[n]$
for which the colouring is the blowup of a colouring of $G$.
Since {$G \arrows{r} H$},
we have a monochromatic canonical copy of $H[t]$.
In fact, this shows that
{$\BlowRam{r}{G}{H}{t}
  \leq \BipRamN{r}{ \BipRamN{r}{\dotsb \BipRamN{r}{t} \dotsb}}$},
that is, an exponential tower of height $e(G)$.
Our first result gives a singly exponential bound for {$\BlowRam{r}{G}{H}{t}$}.

\begin{theorem}
\label{thm:upper_bound_blow}
If {$G \arrows{r} H$},
then {$G[c^t] \barrows{r} H[t]$}
for some constant $c = c(G,H,r)$.
\end{theorem}

In \Cref{sec:lower_bound},
using the Lovász Local Lemma,
we prove a corresponding exponential lower bound of the form
\begin{equation*}
  \BlowRam{r}{G}{H}{t}
  \geq \bigl(1+o(1)\bigr) \bigl(r^{d(H)/2}\bigr)^{t},
\end{equation*}
where $d(H) = 2e(H)/v(H)$ is the average degree of $H$.
\Cref{thm:lower_bound_blow} provides the precise lower bound we obtain,
which recovers the best known bound on $\BipRamN{2}{t}$.

Note that the lower bound we obtain
does not depend on the graph $G$ asymptotically.
If $G$ is large and has many copies of $H$,
it should be harder for a random colouring of $G[n]$
to avoid a canonical monochromatic copy of $H[t]$.
In fact, we conjecture that for all $r \geq 2$
and graphs $H$,
there is a constant $c = c(H,r)$ such that
if {$G \arrows{r} H$},
then {$G[c^t] \barrows{r} H[t]$}.
See \Cref{conj:upper}.

Although we have not managed to settle this conjecture,
not even the case $r = 2$
and $H = K_3$,
we can provide some evidence to support it.
More precisely,
we show that above the threshold for the event
{$\Gnp{n}{p} \arrows{r} H$},
the conjecture holds for $\Gnp{n}{p}$ with high probability.

Let $m_2(H)$ be the $2$-density of a graph $H$
(see \Cref{sec:random_graphs} for a precise definition).
Rödl and Ruciński~\cite{rodl_rucinski1995ramsey_threshold}
proved that $p = n^{-1/m_2(H)}$ is the correct order
for the threshold of the event {$\Gnp{n}{p} \arrows{r} H$}
when $H$ has at least one component that is not a star.
When $H$ is a star forest
and $\Delta(H) \geq 2$,
the threshold occurs at a lower value of $p$,
while if $\Delta(H) = 1$ then there is a coarse threshold at $p = 1/n^2$.

\begin{theorem}
\label{thm:random_graph_blow}
Let $r \geq 2$
and let $H$ be a graph with maximum degree $\Delta(H) \geq 2$.
There are constants $c = c(H,r)$
and $C = C(H,r)$ such that,
if $p \geq Cn^{-1/m_2(H)}$ then
\begin{equation*}
  \lim_{n \to \infty}
  \prob \bigl( \Gnp{n}{p} [c^t] \barrows{r} H[t] \bigr) = 1.
\end{equation*}
\end{theorem}

The main tool used in the proof of \Cref{thm:upper_bound_blow}
is a powerful theorem of Nikiforov~\cite{nikiforov2008blowup_cliques}
concerning blowups (see \Cref{thm:nikiforov}).
In \Cref{sec:upper_bound} we state an improvement of his theorem
under an additional assumption,
which allows us to obtain a better constant $c$ in \Cref{thm:upper_bound_blow}.
Moreover, Nikiforov's theorem implies that one of the parts
of the monochromatic blowup of $H$
can be of size exponential in $t$,
see \Cref{thm:upper_bound_blow_bigger}.

We prove \Cref{thm:random_graph_blow} in \Cref{sec:random_graphs},
using the hypergraph container method of
Balogh, Morris and Samotij~\cite{balogh_morris_samotij2015containers}
and Saxton and Thomason~\cite{saxton_thomason2015containers}.
More specifically,
we use a container theorem for sparse sets in $H$-free graphs,
stated explicitly by Saxton and Thomason.

In \Cref{sec:conj},
we propose some conjectures and open problems,
including the aforementioned \Cref{conj:upper}.
We also examine a family of minimal graphs
with the property that {$G \arrows{r} H$},
but for which the bound given by our proof of \Cref{thm:upper_bound_blow}
is not strong enough to deduce \Cref{conj:upper} in this case.
The construction of this family,
due to Burr, Erd\H{o}s and Lovász~\cite{burr_erdos_lovasz1976ramsey_type},
uses the so-called signal senders.


\section{Upper Bound}
\label{sec:upper_bound}

In this Section,
we prove a quantified version of \Cref{thm:upper_bound_blow}.
Before stating this result precisely,
we introduce some concepts and some notation.

Let $G$ be a graph on $n$ vertices.
Given $t_1, \dotsc, t_n$ positive integers,
we define the $(t_1,\dotsc,t_n)$-\emph{blowup} of $G$,
denote by $G[t_1,\dotsc,t_n]$,
as the graph obtained from $G$ by replacing each vertex $v_i$
with an independent set $U_i$ of $t_i$ vertices.
For every edge $v_i v_j$ in $G$,
we put a complete bipartite graph between $U_i$ and $U_j$.
The $t$-blowup of $G$ is the graph $G[t] = G[t,\dotsc,t]$.

For graphs $G$ and $H$,
we define the $r$-\emph{multiplicity} of $H$ in $G$
as the minimum number of monochromatic copies of $H$
over all $r$-colourings of the edges of $G$,
and denote that quantity by $\Mult{r}{H}{G}$.
Note that the statement that {$G \arrows{r} H$}
is then equivalent to the statement that $\Mult{r}{H}{G} \geq 1$.
Also, note that $\Mult{1}{H}{G}$ is the number of copies of $H$ in $G$.
We call the ratio
\begin{equation*}
  \RamRob{r}{H}{G} \defined \frac{\Mult{r}{H}{G}}{\Mult{1}{H}{G}}
\end{equation*}
the \emph{Ramsey $r$-robustness} of $H$ in $G$.
Thus, $\RamRob{r}{H}{G}$ is the minimum proportion
of monochromatic copies of $H$ in $G$
that we can guarantee that appears
in any $r$-colouring of the edges of $G$.
Note that $\RamRob{r}{H}{G} > 0$
if, and only if, {$G \arrows{r} H$}.
We prove the following theorem,
which implies \Cref{thm:upper_bound_blow},
and which we also use to prove \Cref{thm:random_graph_blow}.

\begin{theorem}
\label{thm:upper_bound_blow_bigger}
If {$G \arrows{r} H$},
then {$G[c^t] \barrows{r} H[t]$},
where $c$ is given by
\begin{equation}
\label{eq:upper_bound_constant}
  c = \exp \left(
    \frac{ r^{v(H)} 4^{v(H)^2 - v(H)} }
      { \RamRob{r}{H}{G}^{v(H)} }
    \right).
\end{equation}
\end{theorem}

Actually,
we are going to prove the stronger statement that
{$G[c^t] \barrows{r} H[t,\dotsc,t,c_0^t]$},
where
\begin{equation*}
c_0 = c^{1-(r^{-1} \RamRob{r}{H}{G})^{v(H)-1}} .
\end{equation*}
This is a strengthening of \Cref{thm:upper_bound_blow},
since $c_0^t \geq t$.
Note that,
via a relabelling of the vertices of $H$,
we can choose which vertex class receives the larger part.

The main ingredient in the proof of the upper bound
is a variant of the following beautiful theorem of
Nikiforov~\cite{nikiforov2008blowup_cliques,nikiforov2008blowup_general}.

\begin{theorem}
\label{thm:nikiforov}
Let $H$ be a graph with $k \geq 2$ vertices.
Let $G$ be a graph with $n$ vertices
and $\rho < 1/4$.
If $G$ contains at least $\rho n^k$ copies of $H$,
then $G$ contains a copy of the blowup
$H[t,\dotsc,t,n^{1-\rho^{k-1}}]$,
where $t = \lfloor \rho^{k^2} \log n\rfloor$. \qed
\end{theorem}

\Cref{thm:nikiforov} was used by
Nikiforov and Rousseau~\cite{nikiforov_rousseau2009goodness}
to resolve several problems of
Burr and Erd\H{o}s~\cite{burr_erdos1983generalizations}
about Ramsey goodness.
This theorem is central to our result,
and we could apply it without any modifications.
Indeed, by applying \Cref{thm:nikiforov},
we would obtain the following constant
$c$ in \Cref{thm:upper_bound_blow_bigger}:
\begin{equation*}
  c = \exp \left(
    \frac{r^{v(H)}v(H)^{v(H)^2}}
      {\RamRob{r}{H}{G}^{v(H)}}
    \right).
\end{equation*}
We can obtain the constant $c$
given in \Cref{thm:upper_bound_blow_bigger}
by applying the following variant of Nikiforov's theorem,
\Cref{thm:nikiforov}.

\begin{theorem}
\label{thm:nikiforov_blowup}
Let $H$ be a graph with $k \geq 2$ vertices.
If $G$ is a subgraph of $H[n]$
with $\rho n^k$ canonical copies of $H$,
then it has a canonical copy of
$H[t,\dotsc,t,n^{1-\rho^{k-1}}]$
where $t = \lfloor \rho^{k} 4^{-k^2+k} \log n \rfloor$.
\end{theorem}

The proof of \Cref{thm:nikiforov_blowup}
is very similar to the original proof of Nikiforov,
and is therefore postponed to \Cref{sec:appendix}.
We now deduce \Cref{thm:upper_bound_blow_bigger}:

\begin{proof}[Proof of Theorem~\ref{thm:upper_bound_blow_bigger}]
Let $G$ and $H$ be graphs with {$G \arrows{r} H$}.
Consider an $r$-colouring of the edges of $G[n]$.
Note that there are $n^{v(G)}$ canonical copies of $G$ in $G[n]$.
Each one of these copies has $\Mult{r}{H}{G} \geq 1$
canonical monochromatic copies of $H$.
But a canonical copy of $H$ can appear in $n^{v(G) - v(H)}$
distinct copies of $G$ in $G[n]$.
Therefore the number of distinct canonical monochromatic
copies of $H$ in $G[n]$ is at least
\begin{equation*}
  \frac{ n^{v(G)}\Mult{r}{H}{G} }{ n^{v(G) - v(H)} }
    = \Mult{r}{H}{G} n^{v(H)}.
\end{equation*}

Thus, there is a colour $i \in [r]$ such that
there are at least $n^{v(H)}\Mult{r}{H}{G}/r$
copies of $H$ in colour $i$.
Furthermore, every canonical copy of $H$ in $G[n]$
correspond to a copy of $H$ in $G$.
There are $\Mult{1}{H}{G}$ copies of $H$ in $G$,
and hence, there are
\begin{equation*}
  L \defined
    \left( \frac{\Mult{r}{H}{G}}{r\Mult{1}{H}{G}} \right) n^{v(H)}
    = r^{-1}\RamRob{r}{H}{G} n^{v(H)}
\end{equation*}
canonical copies of $H$ in $G[n]$,
of colour $i$,
all corresponding to the same copy of $H$ in $G$.
In other words,
there is some copy $H'$ of $H$ in $G$ such that
there are $L$ canonical copies of $H$ of colour $i$ in $H'[n]$.

Let $H'[n]_{(i)}$ be the subgraph of $H'[n]$
of the edges of colour $i$.
We apply \Cref{thm:nikiforov_blowup} to $H'[n]_{(i)}$
with $\rho = r^{-1}\RamRob{r}{H}{G}$
to obtain a copy of
$H[t,\dotsc,t,n^{1-\rho^{v(H)-1}}]$
of colour $i$ in $G[n]$,
where
\begin{equation*}
  t = r^{-v(H)}\RamRob{r}{H}{G}^{v(H)} 4^{-v(H)^2+v(H)} \log n.
\end{equation*}

Consequently,
if we take
$c = c(G,H,r)
  = \exp (r^{v(H)} \RamRob{r}{H}{G}^{-v(H)} 4^{v(H)^2 - v(H)} )$,
we have that {$G[c^t] \barrows{r} H[t,\dotsc,t,c_0^t]$},
for $c_0 = c^{1-(r^{-1} \RamRob{r}{H}{G})^{v(H)-1}}$.
\end{proof}

We point out that whichever version of Nikiforov's theorem is used,
we actually can find a blowup of $H$
with one of the parts of size polynomial in $v(G)$,
instead of logarithmic.
This stronger conclusion is not needed
if we just want to find a monochromatic copy of $H[t]$ in $G[n]$,
but we get a larger part for free by applying this result.
This asymmetric phenomenon seems to appear naturally in extremal
and Ramsey questions in blowup ambient graphs.


\section{Lower Bound}
\label{sec:lower_bound}

In this Section,
we set out to prove the following theorem.

\begin{theorem}
\label{thm:lower_bound_blow}
For $r \geq 2$
and graphs $G$ and $H$,
we have
\begin{equation*}
  \BlowRam{r}{G}{H}{t}
    \geq \bigl(1 + o(1)\bigr)
      \bigl( r^{d(H)/v(H)} \euler^{-1} \bigr) \,t\, r^{d(H)t/2}.
\end{equation*}
\end{theorem}

We recall that $d(H) = 2e(H)/v(H)$ is the average degree of $H$.
To obtain this lower bound,
we produce an $r$-colouring of the edges of $G[n]$ randomly.
For a suitable value of $n$,
we check via the Lovász Local Lemma
that the probability that it has no canonical monochromatic
copy of $H[t]$ is positive.

Let $A_1, A_2 , \dotsc , A_n$
be events in an arbitrary probability space.
A graph $D = (V, E)$
on the set of vertices $V = \{1, 2, \ldots , n\}$
is called a dependency graph for the events $A_1 , \dotsc , A_n$
if for each $i$,
$1 \leq i \leq n$,
the event $A_i$ is mutually independent of all the events
$\{ A_j : \{i,j\} \not\in E \}$.
As we want to avoid the same graph in all the colours,
the simpler symmetric version of the Lovász Local Lemma is sufficient.
We use the following version of the local lemma,
see~\cite[Corollary 5.1.2]{alon_spencer2016book}:

\begin{lemma}[Lovász Local Lemma]
\label{lemma:lll}
Suppose that $D = (V, E)$ is a dependency graph
for the events $A_1, A_2, \dotsc, A_n$.
Suppose that $D$ has degree bounded by $d$
and that $\prob (A_i) \leq p$ for all $1 \leq i \leq n$.
If $\euler p(d+1)\leq 1$,
then $\prob (\cap_{i=1}^n \comp{A_i}) > 0$.
\end{lemma}

In the next proposition,
we apply the local lemma to provide a condition
that produces good lower bounds for blowup Ramsey numbers
in a very general setting.
Indeed, we will produce colourings that
avoid a monochromatic blowup of $H$
with vertex classes of distinct sizes.

Before we state the proposition,
we recall that a graph homomorphism $\phi \from H \to G$
is a map from the vertices of $H$ to the vertices of $G$
such that if $u \sim_H v$
then $\phi(u) \sim_G \phi(v)$.
We denote by $\inj(H,G)$ the number of injective homomorphism from $H$ to $G$.
An injective homomorphism from $H$ to $G$
can be thought as an copy of $H$ in $G$,
but you also keep track of which vertex of $H$
corresponds to which vertex of $G$.

\begin{proposition}
\label{prop:lower_bound_lll}
Let $G$ and $H$ be graphs
and $t_1, \dotsc, t_{v(H)}$ positive integers.
Set $\tilde H = H[t_1, \dotsc, t_{v(H)}]$
and $\Delta = \max\{ t_i t_j : i \sim_H j \}$.
Then {$G[n] \notbarrows{r} \tilde H$},
given that
\begin{equation}
\label{eq:lower_bound_lll_cond}
  \euler \inj(H,G) e(\tilde H)
  r^{1 - e(\tilde H)} \frac{\Delta}{n^2}
  \prod_{w \in [k]} \binom{n}{t_w}
    \leq 1.
\end{equation}
\end{proposition}
\begin{proof}
Consider the blowup $G[n]$,
where for each vertex $j$,
we associate a vertex class $V_i$ of size $n$.
Consider a random uniform $r$-colouring of the edges of $G[n]$
and define,
for each canonical copy of $\tilde H$ in $G$,
the event that such copy is monochromatic.
Each one of these events have the same probability
$p = r^{1 - e(\tilde H)}$.
Furthermore,
each event is mutually independent from all other events
whose corresponding copy of $\tilde H$ is edge-disjoint
with its own copy of $\tilde H$.

Having fixed a copy of $\tilde H$ in $G[n]$
and an edge $i \sim j$ in $H$,
we bound the number of canonical copies of $\tilde H$
that has at least one edge in common between the vertex classes
$V_{\phi(i)}$ and $V_{\phi(j)}$ by
\begin{equation}
\label{eq:bound_choice}
  t_i t_j \inj(H,G) \binom{n-1}{t_u-1} \binom{n-1}{t_v - 1}
  \prod_{w \neq u,v} \binom{n}{t_w}
    = \frac{\inj(H,G) t_i t_j t_u t_v}{n^2} \prod_{w} \binom{n}{t_w}.
\end{equation}

 There are $t_i t_j$ choices for the intersecting edge.
 At most $\inj(H,G)$ choices for a homomorphism $\varphi$
 with $\varphi(u) = \phi(i)$ and $\varphi(v) = \phi(j)$
 and the remaining vertices of $\tilde H$ are chosen without restriction.
 Now we sum~(\ref{eq:bound_choice})
 over the possible choices for the edge $i \sim j$:
\begin{equation*}
d \leq \sum_{i \sim j} \frac{\inj(H,G) t_i t_j t_u t_v}{n^2}
    \prod_{w} \binom{n}{t_w}
  \leq \inj(H,G) e(\tilde H) \frac{\Delta}{n^2} \prod_{w} \binom{n}{t_w}.
\end{equation*}

Thus, the condition in~(\ref{eq:lower_bound_lll_cond})
implies that $\euler pd \leq 1$,
so we can apply \Cref{lemma:lll}
and conclude that with positive probability,
none of the events occur.
\end{proof}

Now, we get \Cref{thm:lower_bound_blow}
as a consequence of \Cref{prop:lower_bound_lll}.
We just have take all $t_i$'s equal
and work out asymptotically the best value of $n$ such that
condition~(\ref{eq:lower_bound_lll_cond}) holds.

\begin{proof}[Proof of \Cref{thm:lower_bound_blow}]
Set $t_i = t$ for all $i$,
$\Delta = t^2$
and apply \Cref{prop:lower_bound_lll}.
Condition~(\ref{eq:lower_bound_lll_cond}) translates to
\begin{equation}
\label{eq:lower_bound_sym_cond}
  \euler \inj(H,G)e(H) r^{1 - e(H)t^2} \frac{t^4}{n^2} \binom{n}{t}^{v(H)}
    \leq 1.
\end{equation}

We want to find the largest $n$, as a function of $t$,
such that condition~(\ref{eq:lower_bound_sym_cond}) holds,
since we then have {$\BlowRam{r}{G}{H}{t} > n$}.
We can take $n$ to be at least exponential in $t$,
so we can approximate the binomial coefficients
as $\binom{n}{t} \sim n^t/t!$.
Also recall Stirling's formula
$t! \sim \sqrt{2 \pi t} \; (t/ \euler)^t$.
Thus, it suffices to show,
as $t \to \infty$, that
\begin{equation*}
  \euler \inj(H,G) e(H) r \cdot \frac{t^4}{n^2}
  \left( \frac{\euler^t n^t}{\sqrt{2 \pi t} \; t^t} \right)^{v(H)}
  r^{-e(H)t^2}
  \ll 1.
\end{equation*}
Ignoring terms that are constant in $t$ and regrouping, we have
\begin{equation}
\label{eq:asymptotic}
  \biggl( \underbrace{\frac{t^{(8 - v(H))/2}}{n^2}}_{P_1} \biggr)
  \biggl( \underbrace{\frac{ \euler n}{t r^{d(H)t/2}} }_{P_2} \biggr)^{v(H)t}
  \ll 1.
\end{equation}

Let $P_1$ and $P_2$ be as identified above in~(\ref{eq:asymptotic}).
If for some $\beta > 0$,
we have that $P_1 \ll \beta^{-t}$
and $ P_2 \leq \beta^t$,
then the condition~(\ref{eq:lower_bound_sym_cond})
is satisfied for large enough $t$.
Note that $P_2 \leq \beta^t$ gives us
\begin{equation*}
  n \leq \bigl( \beta^{1/v(H)}\euler^{-1} \bigr) \, t \, r^{d(H) t/2}.
\end{equation*}

Therefore, if $\varepsilon > 0$,
the condition $P_1 \ll \beta^{-t}$ is satisfied
for $\beta = r^{d(H)} - \varepsilon$.
If $n_0$ is the largest $n$ that satisfies
condition~(\ref{eq:lower_bound_sym_cond}),
then we have
\begin{equation*}
  n_0
    \geq \bigl(1 + o(1)\bigr)
    \bigl( (r^{d(H)} - \varepsilon)^{1/v(H)} \euler^{-1} \bigr)
    \,t\, r^{d(H) t/2},
\end{equation*}
for all $\varepsilon > 0$.
Sending $\varepsilon$ to zero, we obtain
\begin{equation*}
  \BlowRam{r}{G}{H}{t}
    \geq \bigl(1 + o(1)\bigr)
    \bigl( r^{d(H)/v(H)} \euler^{-1} \bigr)
    \,t\, r^{d(H) t/2}. \qedhere
\end{equation*}
\end{proof}

As already discussed,
this lower bound does not depend on the graph $G$.
Note that the condition~(\ref{eq:lower_bound_sym_cond})
itself depends on $G$,
but this is lost in the asymptotic behaviour.
If $G$ had no copies of $H$ whatsoever,
then $\inj(H,G) = 0$
and condition~(\ref{eq:lower_bound_sym_cond}) is trivial.
Furthermore, recall that we do not assume that {$G \arrows{r} H$}.

It is important to notice that a similar lower bound
could also be obtained by an application of the first moment method.
In fact, we would obtain the weaker bound:
\begin{equation*}
  \BlowRam{r}{G}{H}{t}
    \geq \bigl(1+o(1)\bigr) \euler^{-1} \,t\, r^{d(H)t/2}.
\end{equation*}
Like in the bipartite case,
the application of the Lovász Local Lemma
provides only a constant improvement,
in this case,
an improvement of $r^{d(H)/v(H)}$.
This is a very minor enhancement over the first moment method,
but it recaptures the lower bound by
Hattingh and Henning~\cite{hattingh_henning1998bipartite_ramsey}.

We observe that \Cref{prop:lower_bound_lll}
can be used to provide a more direct counterpart to
\Cref{thm:upper_bound_blow_bigger} in the following sense.
We can show that for every constant $k > 1$, we have
\begin{equation}
\label{eq:asym_lower_bound}
  G[( t / \euler) (r^{d(v)}k)^t] \notbarrows{r} H[t, \dotsc, t, k^t].
\end{equation}

Indeed, this shows that any method capable of proving that
{$G[c^t] \barrows{r} H[t, \dotsc, t, k^t]$}
is limited to give a relatively weak upper bound on $c$.
Take the case {$K_2 \barrows{2} K_2$} for example.
We know from the bounds on the bipartite Ramsey numbers that
{$K_2[\log_2(t) 2^{t + 1}] \barrows{2} K_2[t]$},
but by applying \Cref{thm:upper_bound_blow_bigger}
we get that {$K_2[e^{64t}] \barrows{2} K_2[t]$}.
This bound is much weaker than the bound we already have,
but what we actually prove is that
{$K_2[e^{64t}] \barrows{2} K_2[t, e^{32t}]$}.
If our target graph is $K_2[t, e^{32t}]$,
then by~(\ref{eq:asym_lower_bound}),
we have {$K_2[e^{32.6 t}] \notbarrows{2} K_2[t, e^{32t}]$}.
This quantifies the limitations of our method,
considering the original aim of finding $n$
such that {$G[n] \barrows{r} H[t]$}.


\section{Random Graphs}
\label{sec:random_graphs}

Before we prove \Cref{thm:random_graph_blow},
let us recall some results concerning
Ramsey properties of random graphs.
Denote by $\Gnp{n}{p}$ the usual model
for a random graph of $n$ vertices
and where each edge is present with probability $p$,
independent from all other edges.
See Bollobás~\cite{bollobas2001book_randomgraph}
for background on random graphs.
In light of \Cref{thm:upper_bound_blow_bigger},
we are interested in finding $m$ such that
{$\Gnp{n}{p}[m] \barrows{r} H[t]$}.
For this to be possible,
it is necessary that {$\Gnp{n}{p} \arrows{r} H$}.
Thus, we first recall some of what is known
about the threshold for this property.

A graph parameter that is useful in the analysis
of Ramsey properties of random graphs is the $2$-density
$m_2(G)$ of a nonempty graph $G$.
First, define $d_2(K_2) = 1/2$
and $d_2(G) = \frac{e(G) - 1}{v(G) - 2}$ otherwise.
The $2$-density of $G$ is defined as
\begin{equation*}
  m_2(G) = \max \left\{ d_2(J) : J \subseteq G, e(J) \geq 1 \right\}.
\end{equation*}

The study of Ramsey properties of random graphs was initiated by Erd\H{o}s.
Answering his question,
Frankl and Rödl~\cite{frankl_rodl1986trianglefree}
showed that if $p \geq n^{-1/2 + \varepsilon}$
then {$\Gnp{n}{p} \arrows{2} K_3$} with high probability.
Later \L{}uczak, Ruciński,
and Voigt~\cite{luczak_rucinski_voigt1992ramsey_random}
(and independently Erd\H{o}s, Sós and Spencer)
proved that $p = n^{-1/2}$ is the threshold
for the event {$\Gnp{n}{p} \arrows{2} K_3$}.
In a series of papers,
Rödl and Ruciński~\cite{rodl_rucinski1993ramsey_lower,
rodl_rucinski1994ramsey_random, rodl_rucinski1995ramsey_threshold}
proved the following remarkable theorem.

\begin{theorem}
\label{thm:ramsey_threshold}
Let $r \geq 2$
and suppose that $H$ is a graph such that
at least one component of $H$ is not a star,
and in the case $r = 2$,
also not a path of length three.
Then, there exist positive constants $c$ and $C$ such that
\begin{equation*}
  \lim_{n \to \infty}
    \prob(\Gnp{n}{p} \arrows{r} H ) =
    \begin{cases}
      0 & \text{if } p \leq c n^{-1/m_2(H)} \\
      1 & \text{if } p \geq C n^{-1/m_2(H)}.
    \end{cases}
\end{equation*}
\end{theorem}

Moreover, Rödl and Ruciński showed that
in the case where $H$ is a star forest,
the threshold for the property {$\Gnp{n}{p} \arrows{r} H$}
is actually $n^{-1 -1/((\Delta(H)-1)r + 1)}$,
where $\Delta(H)$ is the maximum degree of $H$.
This occurs before the $m_2$ threshold as above.
In the case $r = 2$
and $H$ being a forest whose components are stars and $P_3$'s,
with at least one $P_3$,
the 1-statement of \Cref{thm:ramsey_threshold} still holds,
but it is necessary to assume $p \ll n^{-1/m_2(H)} = 1/n$
for the 0-statement.

As we want to show that {$\Gnp{n}{p}[m] \barrows{r} H[t]$}
if $m$ is large enough,
we only need the 1-statement.
It is necessary to assume that $p \geq C n^{-1/m_2(H)}$
when $H$ is not a star forest,
but we assume this condition for all $H$ for simplicity.
In fact, we also assume that $\Delta(H) \geq 2$.
Under these conditions,
the subgraph count $X_H = \Mult{1}{H}{\Gnp{n}{p}}$
is concentrated around its mean,
see Janson, \L{}uczak and
Ruciński~\cite{janson_luczak_rucinski2011random_book}*{Section~3.1}.
This holds since $\Delta(H) \geq 2$ implies $m(H) < m_2(H)$,
where $m(H)$ is the maximum density
$d(J) \defined e(J)/v(J)$ of a subgraph $J \subseteq H$.
In particular,
$X_H \leq 2\expec(X_H)$ holds with high probability.

The best that we can hope
is that there is a canonical monochromatic copy of $H[t]$
in every $r$-colouring of an $m$-blowup of $\Gnp{n}{p}$,
where $m = m(H,r,t)$.
That is, $m$ does not depend on the ambient graph,
only on the graph that we want to find.

In view of \Cref{thm:upper_bound_blow_bigger},
it is sufficient to show that the multiplicity
$\RamRob{r}{H}{\Gnp{n}{p}}$ is bounded away from $0$
with high probability.
To prove this, we adapt the proof of
Nenadov and Steger~\cite{nenadov_steger2016ramsey}
of the 1-statement of \Cref{thm:ramsey_threshold}.
As in their proof,
we use a version of the hypergraph container theorem
of Balogh, Morris and Samotij~\cite{balogh_morris_samotij2015containers}
and Saxton and Thomason~\cite{saxton_thomason2015containers}.
We differ from Nenadov and Steger by using a container for sparse sets,
instead of independent sets.

We introduce some notation.
Let $\mathcal{G}(n)$ be the set of $2^{\binom{n}{2}}$
graphs with vertex set $[n]$.
Let $\mathcal{F}_r(H,n,b)$ be the family defined as
\begin{equation*}
  \mathcal{F}_r(H,n,b)
    = \left\{ G \in \mathcal{G}(n) : \Mult{r}{H}{G} \leq b  \right\}.
\end{equation*}
Note that $\mathcal{F}_1(H,n,0)$ is
the family of $H$-free graphs on $n$ vertices
and $\mathcal{F}_r(H,n,0)$ is
the family of graphs $G$ on $n$ vertices
for which {$G \notarrows{r} H$}.

The specific container theorem we use,
concerning $H$-sparse graphs,
was stated explicitly by Saxton and Thomason.

\begin{theorem}
\label{thm:containers}
Let $H$ be a graph with $e(H) \geq 2$.
For any $\varepsilon > 0$
there exists $n_0$
and $k = k(H,\varepsilon) > 0$
such that the following is true
for all $n \geq n_0$.
For $n^{-1/m_2(H)} \leq q \leq 1/n_0$,
there exist functions
\begin{equation*}
  f \from \mathcal{G}(n) \to \mathcal{G}(n)
  \quad \text{ and } \quad
  g \from \mathcal{F}_1(H,n,q^{e(H)}n^{v(H)}) \to \mathcal{G}(n),
\end{equation*}
such that for all $G \in \mathcal{F}_1(H,n,q^{e(H)}n^{v(H)})$,
\begin{enumerate}[label=(\roman*)]
  \item $e(g(G)) \leq kqn^2$,
  \item $\Mult{1}{H}{f(g(G))} \leq \varepsilon n^{v(H)}$,
  \item $g(G) \subseteq G \subseteq f(g(G))$.
\end{enumerate}
\end{theorem}

This theorem is a consequence
of Theorem~9.2 in~\cite{saxton_thomason2015containers}
(take $\ell = 2$, $\tilde G = G(n,H)$,
and note that (b) implies (ii).
To obtain (i) and (iii), use (a) and (d)
together with Remark~2.2 in~\cite{saxton_thomason2016online_containers},
which allows us to consider a single set as a `signature',
instead of a tuple $T = (T_1, \dotsc, T_s)$).
The following standard lemma is also useful for us.
For a proof,
see Corollary 8.3 in~\cite{frieze_karonski2016random_book},
for instance.

\begin{lemma}
\label{lem:folklore_ramsey}
Let $r \geq 1$ and $H$ be a graph.
Then there are constants
$\delta, \varepsilon > 0$ and $n_0$
such that the following is true
for all $n \geq n_0$.
For any graphs
$G_1, \dotsc, G_r
  \in \mathcal{F}_1(H,n,\varepsilon n^{v(H)})$,
we have
$e(K_n \setminus (G_1 \cup \dotsb \cup G_r))
  \geq \delta n^2$.
\end{lemma}

Now, we can precisely state
and prove the robustness result for $\Gnp{n}{p}$.

\begin{theorem}
\label{thm:random_multiplicity}
Let $r \geq 2$ and let $H$ be a graph
with $\Delta(H) \geq 2$.
There are constants $\eta = \eta(H,r) > 0$,
and $C = C(H,r)$ such that,
if $p \geq Cn^{-1/m_2(H)}$ then
\begin{equation*}
  \lim_{n \to \infty}
    \prob \bigl( \RamRob{r}{H}{\Gnp{n}{p}} \geq \eta \bigr)
    = 1.
\end{equation*}
\end{theorem}
\begin{proof}
Let $\varepsilon = \varepsilon(H,r)$
and $\delta = \delta(H,r)$ be as in \Cref{lem:folklore_ramsey},
and let $n_0 = n_0(H,r)$
and $k = k(H,r)$ be as in \Cref{thm:containers}
for $H$, $\varepsilon$ and $q = \gamma p$,
where $\gamma \geq 0$ is to be determined.
Choose $C = 1/\gamma$ and let $p \geq C n^{-1/m_2(H)}$,
with $p = o(1)$,
and assume that $n \geq n_0$.
We have $p \gg 1/n^2$,
since $m_2(H) \geq 1$ whenever $\Delta(H) \geq 2$.

Consider the event
$\mathcal{E} =
  \bigl\{ \Mult{r}{H}{\Gnp{n}{p}} \leq q^{e(H)}n^{v(H)} \bigr\}$.
If $\mathcal{E}$ holds,
then there exists a colouring $c \from E(\Gnp{n}{p}) \to [r]$
such that for all $i \in [r]$,
the subgraph of edges of colour $i$,
$G_i \defined (\Gnp{n}{p})_{(i)}$,
have few copies of $H$,
namely, $\Mult{1}{H}{G_i} \leq q^{e(H)}n^{v(H)}$.
By \Cref{thm:containers}, for all $i \in [r]$
there is a `signature' graph $S_i \defined g(G_i)$,
such that $S_i \subseteq G_i \subseteq f(S_i)$.
Define the graph
\begin{equation*}
  K(S_1, \dotsc, S_r)
    \defined K_n \setminus (f(S_1) \cup \dotsb \cup f(S_r) ),
\end{equation*}
and note that $\Gnp{n}{p}$ avoids all the edges of $K(S_1, \dotsc, S_r)$.
Hence, by the union bound, we have
\begin{equation*}
  \prob( \mathcal{E} )
  \leq \sum_{(S_1, \dotsc, S_r)} \prob
    \left(
      \begin{array}{@{}c@{}}
        S_1, \dotsc, S_r \subseteq \Gnp{n}{p} \text{ and} \\
        K(S_1, \dotsc, S_r) \subseteq \comp{\Gnp{n}{p}}
      \end{array}
    \right),
\end{equation*}
where $(S_1, \dotsc, S_r)$ runs over all the possible
sequences of signatures given by \Cref{thm:containers}.
Note that $E(S_1) \cup \dotsb \cup E(S_r)$
and $E(K(S_1, \dotsc ,S_r))$ are disjoint sets of edges,
and hence the events $S_1, \dotsb,  S_r \subseteq \Gnp{n}{p}$
and $K(S_1, \dotsc, S_r) \subseteq \Gnp{n}{p}^c$ are independent.

Since $\Mult{1}{H}{f(S_i)} \leq \varepsilon n^{v(H)}$,
\Cref{lem:folklore_ramsey} implies that
$e(K(S_1, \dotsc, S_r)) \geq  \delta n^2$.
Defining $S^+ \defined \bigcup_{i \in [r]}E(S_i)$,
we can bound
\begin{align*}
  \prob( S_1, \dotsc, S_R \subseteq \Gnp{n}{p} )
    &\leq p^{ |S^+| }, \text{ and} \\
  \prob( K(S_1, \dotsc, S_r) \subseteq \comp{\Gnp{n}{p}} )
    &\leq (1-p)^{\delta n^2} \leq \exp(-\delta p n^2).
\end{align*}

Now, the whole sum can be bounded via the following strategy.
We sum over the possible values
of $s \defined |S^+| \leq r k q n^2$.
First, we choose $s$ edges in $K_n$
to correspond to the union $S^+$,
so we have $\binom{\binom{n}{2}}{s}$ choices.
Next, for each edge,
we choose on which signatures they will appear.
Since the signatures are disjoint,
each edge has $r$ choices,
so we have at most $r^s$ possibilities in total.
Thus, we have
\begin{align*}
  \prob( \mathcal{E} )
    &\leq \exp(-\delta p n^2)
      \sum_{s \leq rkqn^2} \binom{\binom{n}{2}}{s} r^s p^s \\
    &\leq \exp(-\delta p n^2)
      \sum_{s \leq rkqn^2} \left( \frac{ \euler r p n^2 }{2s}\right)^s.
\end{align*}
Observing that $x \mapsto (A/x)^x$
is increasing on the interval $(0,A/\euler)$,
as long as $M < A/ \euler$,
we can bound the sum $\sum_{x \leq M} (A/x)^x$
by $M (A/M)^M$.
Recall that $q = \gamma p$,
and by choosing $\gamma = \gamma(H,r)$ sufficiently small
with respect to $k = k(H,r)$,
and consequently,
choosing $C = 1/\gamma$ large enough,
we have
\begin{equation*}
  \sum_{s \leq rkqn^2} \left( \frac{\euler r p n^2 }{2 s} \right)^s
    \leq rk \gamma p n^2
      \left( \frac{\euler}{2 k \gamma} \right)^{rk\gamma  pn^2}
    \leq \exp(\delta p n^2 /2).
\end{equation*}

In the end, we obtain
$\prob( \mathcal{E} )
  \leq \exp(-\delta n^2p /2) = o(1)$,
since $p \gg 1/n^2$.
Now, consider the random variable
$X_H = \Mult{1}{H}{\Gnp{n}{p}}$.
We have seen above that $X_H \leq 2 \expec(X_H)$
holds with high probability in the range we consider.
Since $\expec(X_H) \leq n^{v(H)}p^{e(H)}$,
we have that
\begin{equation*}
  \RamRob{r}{H}{\Gnp{n}{p}}
    = \frac{ \Mult{r}{H}{\Gnp{n}{p}} }{ \Mult{1}{H}{\Gnp{n}{p}} }
    \geq \frac{q^{e(H)}n^{v(H)}}{2 \expec[X_H]}
    \geq \frac{\gamma^{e(H)}}{2}
\end{equation*}
holds with high probability.
Thus, choosing $\eta = \gamma^{e(H)}/2$,
we are done.
\end{proof}

If $\Delta(H) \leq 1$,
then $m_2(H) = m(H) = 1/2$,
thus the threshold function for {$\Gnp{n}{p} \arrows{r} H$} is $p = 1/n^2$.
In this range, the number of edges in $\Gnp{n}{p}$
converges in distribution to a Poisson random variable,
so in particular,
there is a positive probability that $\Gnp{n}{p}$ is empty.
Therefore, \Cref{thm:random_multiplicity} cannot hold
for all $p \geq C/n^2$ in such cases.
If we instead assume $p \gg 1/n^2$,
it is easy to see that {$\Gnp{n}{p} \arrows{r} H$} with high probability.
For completeness,
we observe that \Cref{thm:random_multiplicity}
implies \Cref{thm:random_graph_blow}.

\begin{proof}[Proof of \Cref{thm:random_graph_blow}]
By \Cref{thm:random_multiplicity},
$\RamRob{r}{H}{\Gnp{n}{p}} \geq \eta$ with high probability
for some $\eta = \eta(H,r) > 0$.
In particular, {$\Gnp{n}{p} \arrows{r} H$},
so by \Cref{thm:upper_bound_blow_bigger},
{$\Gnp{n}{p}[c^t] \barrows{r} H[t]$} for
\begin{equation*}
  c = \exp \left(
      \frac{ r^{v(H)} 4^{v(H)^2 - v(H)} }
        { \RamRob{r}{H}{\Gnp{n}{p}}^{v(H)} } \right)
    \leq \exp \left(
      \frac{ r^{v(H)} 4^{v(H)^2 - v(H)} }
        { \eta^{v(H)} } \right),
\end{equation*}
an upper bound that is a function of $H$
and $\eta = \eta(H,r)$ only.
\end{proof}

In particular, we have
{$\BlowRam{r}{\Gnp{n}{p}}{H}{t} \leq c^t$}
with high probability
for a constant $c = c(H,r)$,
whenever $p \geq C n^{-1/m_2(H)}$.


\section{Conjectures and Open Problems}
\label{sec:conj}

Throughout the previous sections,
we considered the problem of finding
a \emph{canonical} monochromatic copy of $H[t]$
in a $r$-colouring of $G[n]$,
given that {$G \arrows{r} H$}.
One may ask what happens if we allow
for a \emph{non-canonical} copy of $H[t]$.
In this case, the condition that
{$G \arrows{r} H$} is no longer necessary.
For instance, take $H = C_5$.
Since $C_5[t]$ is a subgraph of $K_3[2t]$,
it also suffices that {$G \arrows{r} K_3$}.
The growth of the associated Ramsey numbers
will also be a single exponential
by \Cref{thm:upper_bound_blow_bigger}.
A corresponding lower bound via
the Lovász Local Lemma should work in principle,
but one would have to consider the many non-canonical ways
that you can embed $C_5[t]$ into $G[n]$,
as well as their intersections.

We now come back to the original problem with canonical copies.
As already noted in \Cref{sec:intro},
the lower bound on {$\BlowRam{r}{G}{H}{t}$}
we obtain in \Cref{thm:lower_bound_blow}
does not depend on $G$ asymptotically.
Additionally, \Cref{thm:random_graph_blow}
implies an upper bound on {$\BlowRam{r}{G}{H}{t}$}
that does not depend on $G = \Gnp{n}{p}$
and $p \geq Cn^{-1/m_2(H)}$.
Given this evidence,
we conjecture that one could find
exponential upper bounds that are uniform on $G$
with {$G \arrows{r} H$}.

\begin{conjecture}
\label{conj:upper}
Let $r \geq 2$
and let $H$ be a graph.
There is a constant $c = c(H,r)$
such that if {$G \arrows{r} H$},
then {$G[c^t] \barrows{r} H[t]$}.
\end{conjecture}

In our efforts to establish \Cref{conj:upper},
we obtained \Cref{thm:upper_bound_blow_bigger}.
Note that if {$G \arrows{r} H$}
has a subgraph $G' \subseteq G$
such that {$G' \arrows{r} H$},
we could apply \Cref{thm:upper_bound_blow_bigger} to $G'$
and obtain a constant $c$ as a function of $G'$.
This shows that to certify \Cref{conj:upper},
we can consider only graphs $G$
that are minimal with respect to the Ramsey property {$G \arrows{r} H$}.
Indeed, one may be tempted to show that
the robustness $\RamRob{r}{H}{G}$ is bounded away from zero
among the class of minimal graphs
{$\mathcal{M}_r(H)
  = \{ G : G \text{ minimal such that } G \arrows{r} H \}$}.
Indeed, it would be enough that
\begin{equation*}
  \inf \{ \RamRob{r}{H}{G} : G \in \mathcal{M}_r(H) \} > 0.
\end{equation*}

We will show that this is not the case in general.
For some graphs $H$,
we show that there exist minimal graphs $G \in \mathcal{M}_r(H)$
with arbitrarily low robustness.
To be able to construct such graphs,
we recall the concept of signal senders.

\begin{definition}
For a graph $H$,
a positive (negative) signal sender
$S = S^+(r, H, e, f )$
($S = S^-(r, H, e , f) $)
is a graph containing
distinguished edges $e$ and $f$
with the following properties:
\begin{enumerate}[label=(\roman*)]
  \item {$S \notarrows{r} H$},
  \item For every $r$-colouring of $S$
    in which there is no monochromatic copy of $H$,
    the edges $e$ and $f$ have the same (distinct) colours.
\end{enumerate}
\end{definition}

Burr, Erd\H{o}s and Lovász~\cite{burr_erdos_lovasz1976ramsey_type}
proved that signal senders exist
when $H$ is a clique in the case $r = 2$.
This was generalised to 3-connected graphs $H$ and $r=2$
by Burr, Ne\v{s}et\v{r}il and Rödl~\cite{burr_nesetril_rodl1985signal_senders},
and finally to 3-connected graphs $H$ and $r \geq 2$
by Rödl and Siggers~\cite{rodl_siggers2008ramsey_minimal}.
Moreover, if a negative signal sender exists for $H$,
one can concatenate two of them to obtain a positive signal sender for $H$.
By repeated concatenation of positive signal senders,
one can obtain signal senders where
the special edges $e$ and $f$ are arbitrarily far away in graph distance.
Using signal senders,
it is possible to construct explicit Ramsey-minimal graphs.

\begin{proposition}
\label{prop:low_robustness}
For integers $t,m \geq 2$,
there exists a graph $G$
that is minimal with the property that {$G \arrows{2} K_{t+1}$},
and with multiplicities $\Mult{2}{K_{t+1}}{G} = 1$
and $\Mult{1}{K_{t+1}}{G} \geq m$.
\end{proposition}
\begin{proof}
We follow the construction
of Burr, Erd\H{o}s and Lovász~\cite{burr_erdos_lovasz1976ramsey_type}
of a family of Ramsey minimal graphs for cliques and two colours.
Let $F$ be the complete graph $K_{t^2 +1}$
and let $v$ be one of its vertices.
Consider $t$ disjoint cliques $Q_1, \dotsc, Q_t$ in $F-v$,
each of size $t$.
We construct the graph $G$ as follows.
Start with a copy of $F$,
together with an edge $e$ disjoint from $F$.
For every edge $f = \{x,y\}$
with $x \in Q_i$, $y \in Q_j$, $i \neq j$,
add a disjoint minimal positive signal sender $S^+(r,K_{t+1},e,f)$.
We require the distance from $e$ to $f$ to be at least $3$,
in order to guarantee that there is no copy of $K_{t+1}$ in $G$
other than the copies completely inside $F$
or those completely inside one of the signal senders.

Now, note that {$G \arrows{2} K_{t+1}$}.
Indeed, suppose that we have a red--blue colouring
of the edges of $G$ without monochromatic $K_{t+1}$.
Additionally, suppose that $e$ is red.
By the definition of the signal senders,
all the edges between $Q_i$'s are also red.
This forces all $Q_i$ to be blue cliques.
Now consider the edges incident to $v$.
If $v$ sends only blue edges to some $Q_i$,
then $Q_i$ with $v$ forms a blue $K_{t+1}$.
On the other hand,
if $v$ sends a red edge to each of the $Q_i$,
then $v$ together with these red neighbours forms a red $K_{t+1}$,
so {$G \arrows{2} K_{t+1}$}.
Furthermore, if $v$ sends exactly one red edge to each one of the $Q_i$,
then this colouring has precisely one monochromatic $K_{t+1}$,
so $\Mult{2}{K_{t+1}}{G} = 1$.

It is not hard to see that $G$ is also Ramsey minimal.
Just note that the signal senders we added
are minimal with respect to the signal sender property,
so if the remove an edge from it,
we can find a colouring $c$ avoiding monochromatic $K_{t+1}$
and with $c(e) \neq c(f)$.

Consider now the family $G_n$ of graphs that are constructed just as $G$,
but we replace one of the minimal positive signal senders $S$
from $e$ to $f$ by a concatenation of $n$ more such signal senders.
More precisely,
we add $n$ new edges $g_1, \dotsc, g_n$ to $G$
and then add internally disjoint copies of $S$
from $e$ to $g_1$, from $g_i$ to $g_{i+1}$ and from $g_n$ to $f$.
Such concatenation is also a minimal positive signal sender.
Furthermore,
the argument above repeats without modifications
and shows that $\Mult{2}{K_{t+1}}{G_n} = 1$
and that $G_n$ is Ramsey minimal.
Finally, observe that
$\Mult{1}{K_{t+1}}{G_n}
  = \Mult{1}{K_{t+1}}{G} + n\Mult{1}{K_{t+1}}{S}
  \to \infty$.
\end{proof}

In particular,
there are graphs $G$ with arbitrarily low robustness
$\RamRob{2}{K_{t+1}}{G}$.
This shows that \Cref{thm:upper_bound_blow_bigger}
is not enough to settle \Cref{conj:upper} in general.
Even so,
\begin{equation*}
\inf \{ \RamRob{r}{H}{G} : G \in \mathcal{M}_r(H) \}
\end{equation*}
can be positive for some graphs $H$.
An easy case for that is when $H$ is $r$-Ramsey-finite,
that is, the family $\mathcal{M}_r(H)$ is finite.
It can be shown from \Cref{thm:ramsey_threshold},
however, that $H$ is not $r$-Ramsey-finite for all $r$
whenever $H$ has a cycle.
For $r = 2$, it is known for instance
that $H$ Ramsey-finite when $H$ is a star with an odd number of
edges~\cite{burr_erdos_faudree_rousseau_schelp1981ramsey_minimal_starforests}
and when $H$ is a matching~\cite{burr_erdos_faudree_schelp1978ramsey_finite}.
We conjecture that the only possible reason
for $\RamRob{r}{H}{G}$ to be bounded away from zero is when
$\mathcal{M}_r(H)$ is finite.

\begin{conjecture}
If $H$ is $r$-Ramsey-infinite,
then
\begin{equation*}
  \inf \{ \RamRob{r}{H}{G} : G \in \mathcal{M}_r(H) \} = 0.
\end{equation*}
\end{conjecture}

The study of this quantity fits well within
the framework of minimisation problems on Ramsey-minimal graphs.
Given a graph parameter $F$,
one can investigate the quantity
\begin{equation*}
  \inf \{ F(G) : G \in \mathcal{M}_r(H) \}.
\end{equation*}
When $F$ is the number of vertices,
we obtain the Ramsey numbers $\Ram{r}{H}$,
and when $F$ is the number of edges,
size-Ramsey numbers
$\SizeRam{r}{H}$~\cite{erdos_faudree_rousseau_schelp1978size_ramsey}.
Other parameters such as chromatic number,
minimal and maximum degree have also been
studied~\cite{burr_erdos_lovasz1976ramsey_type}.

The study of the Ramsey multiplicities is interesting in itself.
Goodman~\cite{goodman1959mono_triangles} initiated the subject\
by determining $\Mult{2}{K_3}{K_n}$ precisely for all $n$.
A survey of Burr and Rosta~\cite{burr_rosta1980ramsey_multiplicity}
collects several results on multiplicities.
They show that $\RamRob{r}{H}{K_n}$ is monotonously nondecreasing in $n$
and bounded above by $r^{1-e(H)}$.
Thus, it is natural to define the
\emph{Ramsey multiplicity constant} of $H$
as the following converging limit
\begin{equation*}
  C_r(H) = \lim_{n \to \infty} \RamRob{r}{H}{K_n}.
\end{equation*}
Erd\H{o}s~\cite{erdos1962number_subgraphs}
conjectured that $C_2(K_t) = 2^{1-\binom{t}{2}}$
and Goodman's result implies that this is true for $t = 3$.
Burr and Rosta further conjectured that $C_2(H) = 2^{1-e(H)}$
for all graphs $H$.
Erd\H{o}s' conjecture was later disproved
by Thomason~\cite{thomason1989disproof_erdos}
for $H = K_{t}$, $t \geq 4$.

Similar questions can be raised for different ambient graphs
in place of the complete graph $K_n$.
Erd\H{o}s and Moon~\cite{erdos_moon1964subgraphs_bip},
for instance, have shown that
\begin{equation*}
  \lim_{\substack{n \to \infty \\ m \to \infty}}
    \RamRob{2}{K_{a,b}}{K_{n,m}}
    = 2^{1-ab},
\end{equation*}
where they considered copies of $K_{a,b}$ in $K_{n,m}$
where the part of size $a$ is sitting inside the part of size $n$.
This confirms Erd\H{o}s conjecture in the bipartite setting.
Another natural setting to consider is a random graph.
\Cref{thm:random_multiplicity} shows that
$\RamRob{r}{H}{\Gnp{n}{p}} \geq \eta > 0$
with high probability,
given that $p \geq Cn^{-1/m_2(H)}$.

Finally, we note that while the upper bound on the blowup Ramsey numbers
provided on \Cref{thm:upper_bound_blow_bigger} is of the form $c^t$,
the constant $c$ can be quite large.
It is natural to ask for more effective upper bounds on these numbers.
We find the following to be specially interesting.

\begin{problem}
What is the smallest $n$
such that {$K_6[n] \barrows{2} K_3[t]$}?
\end{problem}

All we know at the moment is that
\Cref{thm:upper_bound_blow_bigger,thm:lower_bound_blow}
imply the weak bounds:
$2^t \leq n \leq e^{(3.3 \times 10^7)t}$.


\section*{Acknowledgement}

We would like to thank Rob Morris for his encouragement
and invaluable suggestions.
We also thank both referees for their careful reading and helpful comments.
This research was partially supported by CAPES, Brazil.



\appendix
\section{Proof of Theorem~\ref{thm:nikiforov_blowup}}
\label{sec:appendix}

In this appendix,
we provide a proof of \Cref{thm:nikiforov_blowup}.
Our proof is essentially the same as that of
Nikiforov~\cite{nikiforov2008blowup_cliques,nikiforov2008blowup_general},
with some modifications.
Following his strategy,
we deduce it from a routine lemma,
adjusted for our purposes:

\begin{lemma}
\label{lem:nikiforov}
Let $k \geq 2$,
let $0 < \rho \leq 1$,
and let $F$ be a bipartite graph with parts $A$ and $B$.
If $e(F) \geq (\rho/2) |A| |B|$
and $\rho |A| /4 + 1
  \geq \lfloor \rho^k 4^{-k^2 + k} \log |B| \rfloor \geq 1$,
then $F$ contains a $K_2[s,t]$
with parts $A_0 \subseteq A$
and $B_0 \subseteq B$,
such that
\begin{equation*}
  s = |A_0|  = \lfloor \rho^k 4^{-k^2 + k} \log |B| \rfloor
  \qquad \text{and} \qquad
  t = |B_0| \geq |B|^{1-\rho^{k-1}}.
\end{equation*}
\end{lemma}
\begin{proof}
Let $m \defined |A|$
and $n \defined |B|$,
and define
\begin{equation*}
  t \defined \max\{ x :
    \text{there exists $K_2[s,x] \subseteq F$ with part of size  $s$ in $A$}
    \}.
\end{equation*}

For any $X \subseteq A$,
write $d(X)$ for the number of vertices
that are neighbours of all vertices of $X$.
For each $X$ with $|X| = s$,
we have $d(X) \leq t$,
thus
\begin{equation}
\label{eq:niki_lemma}
  \sum_{v \in B} \binom{d(v)}{s}
    = \sum_{\substack{X\subset A \\ |X| = s}} d(X)
    \leq t \binom{m}{s}.
\end{equation}

By convexity of function
$x \mapsto \binom{x}{s}\mathbbm{1}_{\{x \geq s-1\}}$,
we obtain
\begin{equation*}
  \sum_{v \in B} \binom{d(v)}{s}
    \geq n \binom{e(F)/n}{s}
    \geq n \binom{\rho m/2}{s}.
\end{equation*}
Combining this inequality with \cref{eq:niki_lemma}, we have
\begin{equation*}
  t \;
    \geq n \binom{\rho m/2}{s}\binom{m}{s}^{-1}
    \geq n \left( \frac{\rho}{4}\right)^s
    \geq n^{1 + \rho^k 4^{-k^2 + k} \log(\rho/4)},
\end{equation*}
where we used that $s \leq \rho m/4 + 1$ on the last step.
Since $\rho \log(4/\rho) < 4^{k^2 - k}$
we have
\begin{equation*}
  t \geq n^{1- \rho^{k-1}}. \qedhere
\end{equation*}
\end{proof}

Before proceeding to the proof of \Cref{thm:nikiforov_blowup},
we introduce some notation.
For a subgraph $G$ of $H[n]$,
denote by $\mathcal{N}(H,G)$ the set of canonical copies of $H$ in $G$.
Given $\mathcal{L} \subseteq \mathcal{N}(H,G)$
and a subgraph $H'$ of $H$,
we write $\mathcal{N}(H',\mathcal{L})$
to denote the set of canonical copies of $H'$
that are contained in some member of $\mathcal{L}$.
Also, given a subgraph $G' \subseteq G$
such that $G' = H[t_1,\dotsc,t_k]$,
we say that a family $\mathcal{L} \subseteq \mathcal{N}(H,G)$
\emph{covers} $G'$
if $E(G') \subseteq \mathcal{N}(K_2,\mathcal{L})$,
that is, the union of the edges of elements of $\mathcal{L}$
covers the graph $G'$,
and there are $\min\{t_1,\dotsc, t_k\}$ disjoint elements of $\mathcal{L}$
as subgraphs of $G'$.

Still assuming that $G$ is a subgraph of $H[n]$,
for a vertex $v \in V(H)$
we denote by $G-v$ the subgraph of $(H-v)[n]$
obtained from $G$ by the removal of the vertex class of $v$ in $H[n]$.
Finally, for any subset $\mathcal{L} \subseteq \mathcal{N}(H,G)$,
and $R \in \mathcal{N}(H-v,G-v)$,
we denote by $d_{\mathcal{L}}(R)$ the number ways that we can extend $R$,
a canonical copy of $H-v$,
to a element of $\mathcal{L}$.

\begin{proof}[Proof of Theorem~\ref{thm:nikiforov_blowup}]
We are going to prove by induction on $k \geq 2$
the following statement:
that every subset $\mathcal{M} \subseteq \mathcal{N}(H,H[n])$
of canonical copies of $H$ in $H[n]$,
with $|\mathcal{M}| \geq \rho n^{k}$,
covers a $H[t,\dotsc,t,n^{1-\rho^{k-1}}]$
with $t = \rho^{k}4^{-k^2+k} \log n$.

For $k = 2$, let $\mathcal{M} \subseteq \mathcal{N}(K_2,K_2[n])$
with $|\mathcal{M}| \geq \rho n^2$
and apply \Cref{lem:nikiforov} to $K_2[n]$.
We obtain that $\mathcal{M}$ covers
a $K_2[\rho^2 4^{-2^2+2} \log n, n^{1- \rho}]$.

Now we proceed to the induction step, with $k > 2$.
Let $G$ be a subgraph of $H[n]$
with $| \mathcal{N}(H,G) | \geq \rho n^{k}$
and let $v \in V(H)$.
The first step is to show that there is
a subset $\mathcal{L} \subseteq \mathcal{N}(H,G)$,
with $|\mathcal{L}| \geq (\rho/2)n^{k}$
such that for all $R \in \mathcal{L}$,
$d_{\mathcal{L}}(R-v) \geq (\rho/2) n$.
We construct this subset via the following procedure:

\begin{algorithmic}
\State $\mathcal{L} \gets \mathcal{N}(H,G)$.
\While {there exists an $R \in \mathcal{L}$
        with $d_{\mathcal{L}}(R-v) < (\rho/2) n$}
  \State $\mathcal{L} \gets \mathcal{L} \setminus \{ R' \in \mathcal{L}
          : R' \text{ is an extension  of } R-v \}$.
\EndWhile
\end{algorithmic}

When it ends,
we have a subset $\mathcal{L}$ with the property that
$d_{\mathcal{L}}(R-v) \geq (\rho/2) n$
for all $R \in \mathcal{L}$.
Also, we have
\begin{equation*}
  |\mathcal{L}|
    > |\mathcal{N}(H,G)| - (\rho/2)n |\mathcal{N}(H-v,G-v)|
    \geq (\rho/2)n^{k}.
\end{equation*}

Now, observe that
$\mathcal{N}(H-v,\mathcal{L}) \subseteq \mathcal{N}(H-v,G-v)$
with
\begin{equation*}
  |\mathcal{N}(H - v, \mathcal{L})|
    \geq |\mathcal{L}| / n
    \geq (\rho/2) n^{k-1}.
\end{equation*}
By the induction hypothesis,
$\mathcal{N}(H - v, \mathcal{L})$ covers a copy of $(H-v)[t']$,
where we have
$t' = \lfloor (\rho/2)^{k-1}4^{-(k-1)^2+k-1} \log n \rfloor$.

Now we build a bipartite graph $F$ with parts $A$ and $B$,
where $A$ is a set of disjoint canonical copies of $H-v$
in the blowup $(H-v)[t']$,
and $B$ is the vertex class of $v$.
This gives us $|A| = t'$, $|B| = n$.
We put an edge between a copy of $H - v$
and a vertex $u \in B$ if together,
they form an element of $\mathcal{L}$.
Therefore,
\begin{equation*}
  e(F)
    \geq d_{\mathcal{L}}(R) |A|
    \geq (\rho/2) n |A|
    = (\rho/2) |A| |B|.
\end{equation*}

We will apply \Cref{lem:nikiforov}
to the bipartite graph $F$.
If $t \defined \lfloor \rho^{k}4^{-k^2+k} \log|B| \rfloor$,
we have to check that $t \leq \rho |A|/4 + 1$.
Indeed
\begin{align*}
  t
    &\leq \rho^{k}4^{-k^2+k} \log n
      \leq (\rho/4) (\rho/2)^{k-1} 4^{-k^2 + k + 1 + (k-1)/2}\log n \\
    &\leq (\rho/4) (\rho/2)^{k-1} 4^{-(k-1)^2 + k - 1}\log n
      \leq \rho |A|/4 + 1,
\end{align*}
where we used that
$-k^2+k+1+(k-1)/2 \leq -(k-1)^2+k-1$
for $k \geq 2$.
Thus, by \Cref{lem:nikiforov},
we have $K_2[t,n^{1-\rho^{k-1}}] \subseteq F$,
with parts $A_0 \subseteq A$ of size $t$
and $B_0 \subseteq B$ of size $n^{1-\rho^{k-1}}$.
Let $H^*$ be the subgraph of $(H-v)[n]$
induced by the union of the members of $A_0$.
For every copy $H'$ of $H-v$ in $A_0$,
it can be joined to any vertex of $u \in B_0$
to form a copy of $H$.
This implies that $H^*$ covers a copy of
$H[|A_0|, \dotsc, |A_0|, |B_0|] = H[t,\dotsc,t,n^{1-\rho^{k-1}}]$.
\end{proof}



\begin{bibdiv}
\begin{biblist}

\bib{alon_spencer2016book}{book}{
      author={Alon, N.},
      author={Spencer, J.},
       title={{The Probabilistic Method}},
     edition={4},
   publisher={Wiley},
        date={2016},
        ISBN={9781119061953},
}

\bib{balogh_morris_samotij2015containers}{article}{
      author={Balogh, J.},
      author={Morris, R.},
      author={Samotij, W.},
       title={Independent sets in hypergraphs},
        date={2015},
     journal={Journal of the American Mathematical Society},
      volume={28},
      number={3},
       pages={669\ndash 709},
}

\bib{beineke_schwenk1976bipartite_ramsey}{article}{
      author={Beineke, L.W.},
      author={Schwenk, A.J.},
       title={On a bipartite form of the {Ramsey} problem},
        date={1976},
     journal={Proceedings of the Fifth British Combinatorial Conference},
      number={XV},
       pages={17\ndash 22},
}

\bib{bollobas2001book_randomgraph}{book}{
      author={Bollobás, B.},
       title={Random graphs},
     edition={2},
      series={Cambridge Studies in Advanced Mathematics},
   publisher={Cambridge University Press},
        date={2001},
}

\bib{burr_erdos1983generalizations}{article}{
      author={Burr, S.A.},
      author={Erd\H{o}s, P.},
       title={Generalizations of a {Ramsey}-theoretic result of {Chv{\'a}tal}},
        date={1983},
     journal={Journal of Graph Theory},
      volume={7},
      number={1},
       pages={39\ndash 51},
}

\bib{burr_erdos_faudree_rousseau_schelp1981ramsey_minimal_starforests}{article}{
      author={Burr, S.A.},
      author={Erd\H{o}s, P.},
      author={Faudree, R.J.},
      author={Rousseau, C.C.},
      author={Schelp, R.H.},
       title={{Ramsey}-minimal graphs for star-forests},
        date={1981},
        ISSN={0012-365X},
     journal={Discrete Mathematics},
      volume={33},
      number={3},
       pages={227\ndash 237},
}

\bib{burr_erdos_faudree_schelp1978ramsey_finite}{inproceedings}{
      author={Burr, S.A.},
      author={Erd\H{o}s, P.},
      author={Faudree, R.J.},
      author={Schelp, R.H.},
       title={A class of {Ramsey}-finite graphs},
        date={1978},
   booktitle={{Proceedings of the 9th SE Conference on Combinatorics, Graph
  Theory, and Computing}},
       pages={171\ndash 178},
}

\bib{burr_erdos_lovasz1976ramsey_type}{article}{
      author={Burr, S.A.},
      author={Erd\H{o}s, P.},
      author={Lovász, L.},
       title={On graphs of {Ramsey} type},
        date={1976},
     journal={Ars Combinatoria},
      volume={1},
      number={1},
       pages={167\ndash 190},
}

\bib{burr_nesetril_rodl1985signal_senders}{article}{
      author={Burr, S.A.},
      author={Nešetřil, J.},
      author={Rödl, V.},
       title={On the use of senders in generalized {Ramsey} theory for graphs},
        date={1985},
     journal={Discrete Mathematics},
      volume={54},
       pages={1\ndash 13},
}

\bib{burr_rosta1980ramsey_multiplicity}{article}{
      author={Burr, S.A.},
      author={Rosta, V.},
       title={On the {Ramsey} multiplicities of graphs -- problems and recent
  results},
        date={1980},
     journal={Journal of Graph Theory},
      volume={4},
      number={4},
       pages={347\ndash 361},
}

\bib{conlon2008bipartite_ramsey}{article}{
      author={Conlon, D.},
       title={A new upper bound for the bipartite {R}amsey problem},
        date={2008},
     journal={Journal of Graph Theory},
      volume={58},
      number={4},
       pages={351\ndash 356},
}

\bib{conlon2009diagonal_ramsey}{article}{
      author={Conlon, D.},
       title={A new upper bound for diagonal {Ramsey} numbers.},
        date={2009},
        ISSN={0003-486X; 1939-8980/e},
     journal={Annals of Mathematics},
      volume={170},
      number={2},
       pages={941\ndash 960},
}

\bib{erdos1947some_remarks}{article}{
      author={Erd\H{o}s, P.},
       title={Some remarks on the theory of graphs},
        date={1947},
     journal={Bulletin of the American Mathematical Society},
      volume={53},
      number={4},
       pages={292\ndash 294},
}

\bib{erdos1962number_subgraphs}{article}{
      author={Erd\H{o}s, P.},
       title={On the number of complete subgraphs contained in certain graphs},
        date={1962},
     journal={Magyar Tud. Akad. Mat. Kutató Int. Közl},
      volume={7},
      number={3},
       pages={459\ndash 464},
}

\bib{erdos_faudree_rousseau_schelp1978size_ramsey}{article}{
      author={Erd\H{o}s, P.},
      author={Faudree, R.J.},
      author={Rousseau, C.C.},
      author={Schelp, R.H.},
       title={The size {Ramsey} number},
        date={1978},
     journal={Periodica Mathematica Hungarica},
      volume={9},
      number={1-2},
       pages={145\ndash 161},
}

\bib{erdos_lovasz1975problems}{article}{
      author={Erd\H{o}s, P.},
      author={Lov{\'a}sz, L.},
       title={Problems and results on 3-chromatic hypergraphs and some related
  questions},
        date={1975},
     journal={Infinite and Finite sets},
      volume={10},
      number={2},
       pages={609\ndash 627},
}

\bib{erdos_moon1964subgraphs_bip}{article}{
      author={Erd\H{o}s, P.},
      author={Moon, J.W.},
       title={On subgraphs of the complete bipartite graph},
        date={1964},
     journal={Canad. Math. Bull},
      volume={7},
      number={1},
       pages={35\ndash 39},
}

\bib{erdos_szekeres1935combinatorial}{article}{
      author={Erd\H{o}s, P.},
      author={Szekeres, G.},
       title={A combinatorial problem in geometry},
        date={1935},
     journal={Compositio Mathematica},
      volume={2},
       pages={463\ndash 470},
         url={http://eudml.org/doc/88611},
}

\bib{frankl_rodl1986trianglefree}{article}{
      author={Frankl, P.},
      author={R{\"o}dl, V.},
       title={Large triangle-free subgraphs in graphs without {$K_4$}},
        date={1986},
     journal={Graphs and Combinatorics},
      volume={2},
      number={1},
       pages={135\ndash 144},
}

\bib{frieze_karonski2016random_book}{book}{
      author={Frieze, A.},
      author={Karoński, M.},
       title={Introduction to random graphs},
   publisher={Cambridge University Press},
        date={2016},
}

\bib{goodman1959mono_triangles}{article}{
      author={Goodman, A.W.},
       title={On sets of acquaintances and strangers at any party},
        date={1959},
        ISSN={00029890, 19300972},
     journal={The American Mathematical Monthly},
      volume={66},
      number={9},
       pages={778\ndash 783},
         url={http://www.jstor.org/stable/2310464},
}

\bib{graham_rodl1987ramsey_numbers}{article}{
      author={Graham, R.L.},
      author={Rödl, V.},
       title={Numbers in {Ramsey} theory},
        date={1987},
     journal={Surveys in Combinatorics},
      volume={123},
       pages={111\ndash 153},
}

\bib{hattingh_henning1998bipartite_ramsey}{article}{
      author={Hattingh, J.H.},
      author={Henning, M.A.},
       title={Bipartite {Ramsey} theory},
        date={1998},
     journal={Utilitas Mathematica},
      volume={53},
       pages={217\ndash 230},
}

\bib{janson_luczak_rucinski2011random_book}{book}{
      author={Janson, S.},
      author={\L{}uczak, T.},
      author={Ruciński, A.},
       title={Random graphs},
   publisher={John Wiley \& Sons},
        date={2011},
      volume={45},
}

\bib{luczak_rucinski_voigt1992ramsey_random}{article}{
      author={\L{}uczak, T.},
      author={Ruciński, A.},
      author={Voigt, B.},
       title={{Ramsey} properties of random graphs},
        date={1992},
     journal={Journal of Combinatorial Theory, Series B},
      volume={56},
      number={1},
       pages={55\ndash 68},
}

\bib{nenadov_steger2016ramsey}{article}{
      author={Nenadov, R.},
      author={Steger, A.},
       title={A short proof of the random {Ramsey} theorem},
        date={2016},
     journal={Combinatorics, Probability and Computing},
      volume={25},
      number={1},
       pages={130\ndash 144},
}

\bib{nikiforov2008blowup_general}{article}{
      author={Nikiforov, V.},
       title={Graphs with many copies of a given subgraph},
        date={2008},
     journal={The Electronic Journal of Combinatorics},
      volume={15},
      number={1},
       pages={6},
}

\bib{nikiforov2008blowup_cliques}{article}{
      author={Nikiforov, V.},
       title={Graphs with many r-cliques have large complete r-partite
  subgraphs},
        date={2008},
     journal={Bulletin of the London Mathematical Society},
      volume={40},
      number={1},
       pages={23\ndash 25},
}

\bib{nikiforov_rousseau2009goodness}{article}{
      author={Nikiforov, V.},
      author={Rousseau, C.C.},
       title={{Ramsey} goodness and beyond},
        date={2009},
        ISSN={1439-6912},
     journal={Combinatorica},
      volume={29},
      number={2},
       pages={227\ndash 262},
         url={https://doi.org/10.1007/s00493-009-2409-2},
}

\bib{ramsey1930original}{article}{
      author={Ramsey, F.P.},
       title={On a problem of formal logic},
        date={1930},
     journal={Proceedings of the London Mathematical Society},
      volume={2},
      number={1},
       pages={264\ndash 286},
}

\bib{rodl_rucinski1993ramsey_lower}{article}{
      author={Rödl, V.},
      author={Ruciński, A.},
       title={Lower bounds on probability thresholds for {Ramsey} properties},
        date={1993},
     journal={Combinatorics, Paul Erd\H{o}s is eighty},
      volume={1},
       pages={317\ndash 346},
}

\bib{rodl_rucinski1994ramsey_random}{article}{
      author={Rödl, V.},
      author={Ruciński, A.},
       title={Random graphs with monochromatic triangles in every edge
  coloring},
        date={1994},
     journal={Random Structures \& Algorithms},
      volume={5},
      number={2},
       pages={253\ndash 270},
}

\bib{rodl_rucinski1995ramsey_threshold}{article}{
      author={Rödl, V.},
      author={Ruciński, A.},
       title={Threshold functions for {Ramsey} properties},
        date={1995},
     journal={Journal of the American Mathematical Society},
      volume={8},
      number={4},
       pages={917\ndash 942},
}

\bib{rodl_siggers2008ramsey_minimal}{article}{
      author={Rödl, V.},
      author={Siggers, M.},
       title={On {Ramsey} minimal graphs},
        date={2008},
     journal={SIAM Journal on Discrete Mathematics},
      volume={22},
      number={2},
       pages={467\ndash 488},
}

\bib{sah2020ramsey}{article}{
      author={Sah, A.},
       title={Diagonal ramsey via effective quasirandomness},
        date={2020},
     journal={arXiv preprint arXiv:2005.09251},
         url={https://arxiv.org/pdf/2005.09251.pdf},
}

\bib{saxton_thomason2015containers}{article}{
      author={Saxton, D.},
      author={Thomason, A.},
       title={Hypergraph containers},
        date={2015},
        ISSN={1432-1297},
     journal={Inventiones mathematicae},
      volume={201},
      number={3},
       pages={925\ndash 992},
         url={https://doi.org/10.1007/s00222-014-0562-8},
}

\bib{saxton_thomason2016online_containers}{article}{
      author={Saxton, D.},
      author={Thomason, A.},
       title={Online containers for hypergraphs, with applications to linear
  equations},
        date={2016},
     journal={Journal of Combinatorial Theory, Series B},
      volume={121},
       pages={248\ndash 283},
}

\bib{spencer1975ramsey_lower}{article}{
      author={Spencer, J.},
       title={{Ramsey}'s theorem — a new lower bound},
        date={1975},
        ISSN={0097-3165},
     journal={Journal of Combinatorial Theory, Series A},
      volume={18},
      number={1},
       pages={108\ndash 115},
  url={http://www.sciencedirect.com/science/article/pii/0097316575900710},
}

\bib{thomason1998upper_ramsey}{article}{
      author={Thomason, A.},
       title={An upper bound for some {Ramsey} numbers},
        date={1988},
     journal={Journal of graph theory},
      volume={12},
      number={4},
       pages={509\ndash 517},
}

\bib{thomason1989disproof_erdos}{article}{
      author={Thomason, A.},
       title={A disproof of a conjecture of {Erd\H{o}s} in {Ramsey} theory},
        date={1989},
     journal={Journal of the London Mathematical Society},
      volume={s2-39},
      number={2},
       pages={246\ndash 255},
}

\end{biblist}
\end{bibdiv}
\end{document}